\documentclass[10pt,letterpaper,twoside]{amsart}
\usepackage[left=1.2in,right=1.2in,top=1in,bottom=1in,headheight=12pt,headsep=18pt]{geometry}
\usepackage{amsmath,amssymb,mathtools,mathrsfs,microtype}
\usepackage{fancyhdr,xcolor}
\definecolor{citeblue}{RGB}{50,155,215}
\usepackage[colorlinks=true,linkcolor=blue,citecolor=citeblue,urlcolor=blue]{hyperref}

\newtheorem{theorem}{Theorem}[section]
\newtheorem{lemma}[theorem]{Lemma}
\newtheorem{proposition}[theorem]{Proposition}
\newtheorem{corollary}[theorem]{Corollary}
\theoremstyle{definition}
\newtheorem{definition}[theorem]{Definition}

\newcommand{\F}{\mathbb F}
\newcommand{\R}{\mathbb R}
\newcommand{\E}{\mathbb E}
\newcommand{\Prob}{\mathbb P}
\newcommand{\PG}{\mathrm{PG}}
\newcommand{\tr}{\operatorname{tr}}
\newcommand{\Res}{\operatorname{Res}}
\newcommand{\norm}[1]{\lVert#1\rVert}
\newcommand{\bx}{\boldsymbol{x}}
\newcommand{\by}{\boldsymbol{y}}
\newcommand{\one}{\mathbf1}
\newcommand{\T}{\mathsf T}
\newcommand{\cS}{\mathscr S}
\newcommand{\cL}{\mathscr L}
\newcommand{\cC}{\mathscr C}

\AtBeginDocument{%
\setlength{\abovedisplayskip}{4pt plus 1pt minus 1pt}%
\setlength{\belowdisplayskip}{4pt plus 1pt minus 1pt}%
\setlength{\abovedisplayshortskip}{3pt plus 1pt minus 1pt}%
\setlength{\belowdisplayshortskip}{3pt plus 1pt minus 1pt}}

\title{Improved Explicit and Algorithmic\\Lower Bounds for $r(5,t)$}
\author{Amrit Kandasamy, Lawrence Zhou, Levi Segal}
\date{}

\hypersetup{pdftitle={Improved Explicit and Algorithmic Lower Bounds for r(5,t)},pdfauthor={Amrit Kandasamy, Lawrence Zhou, Levi Segal}}

\begin{document}

\begin{abstract}
We modify a construction of Brada\v{c} \cite{Bradac26} using a vertex ordering to give, for every prime power $q$, an explicit $K_5$ free graph with more than $q^7$ vertices and independence number below $600q^4$. This shows $$r(5,t)=\Omega(t^{7/4}).$$ Additionally, pruning a graph defined by the Coulter and Matthews polynomial $X^{14}$ \cite{CM}, we give a deterministic polynomial time construction proving
$$r(5,t)=\Omega\bigl(t^{20/11-\varepsilon}\bigr)$$
for every fixed $\varepsilon>0$. These results improve the constructive lower bound $\Omega(t^{5/3})$ of Kostochka, Pudl\'ak and R\"odl \cite{KPR}.
\end{abstract}

\maketitle
\thispagestyle{plain}

\section{Introduction}

The Ramsey number $r(s,t)$ is the least $N$ such that every graph on $N$ vertices contains a $K_s$ or an independent set of size $t$. So $\omega(H)<s$ and $\alpha(H)<t$ give $r(s,t)>|V(H)|$. Kostochka, Pudl\'ak and R\"odl \cite{KPR} proved the constructive lower bound $r(5,t)=\Omega(t^{5/3})$. We give two constructions.

\begin{theorem}\label{explicit}
Let $q\ge2$ be a prime power. There is an explicit graph $\Gamma_q$ such that
$$\omega(\Gamma_q)\le4,\,|V(\Gamma_q)|=(q^4+q^3+q^2+q+1)(q^3+q^2+q+1)$$
and $\alpha(\Gamma_q)<600q^4$. So
$$r(5,t)=\Omega(t^{7/4}).$$
\end{theorem}

The second construction gives exponents approaching $20/11$ by choosing edges deterministically.

\begin{theorem}\label{algorithmic}
For each fixed integer $k\ge3$ and all sufficiently large $t$ we construct in deterministic polynomial time a graph $H$ with
$$\omega(H)\le4,\,\alpha(H)<t,\,|V(H)|=\Omega\bigl(t^{20k/(11k+7)}\bigr).$$
So for every fixed $\varepsilon>0$
$$r(5,t)=\Omega\bigl(t^{20/11-\varepsilon}\bigr).$$
\end{theorem}

We prove Theorem \ref{explicit} in Section \ref{constructive} and Theorem \ref{algorithmic} in Section \ref{deterministic}. We use adjacency matrices with operator norm $\norm{\cdot}$ and let $J$ be the matrix whose entries are all $1$. The two sections use different base graphs.

\section{The constructive bound}\label{constructive}

\subsection{The construction}

We modify a construction of Brada\v{c} \cite{Bradac26} using a vertex ordering. We represent points of $\PG(4,q)$ by $1$ dimensional subspaces of $\F_q^5$. Let $G=G(4,q)$ be the graph with possible loops on $\PG(4,q)$ in which $a,b\in V(G)$ are adjacent when $\boldsymbol{a}\cdot\boldsymbol{b}=0$. For $b\in\PG(4,q)$ let $H_b=\{a\in\PG(4,q):ab\in E(G)\}$. Then $H_b$ is a hyperplane of $\PG(4,q)$. Let
$$n=|V(G)|=q^4+q^3+q^2+q+1,\,d=|H_b|=q^3+q^2+q+1.$$
The digraph $D^*=D^*(4,q)$ of \cite{Bradac26} is defined by
$$V(D^*)=\{(a,b):ab\in E(G)\}$$
and
$$((a,b),(a',b'))\in A(D^*)\iff ab'\in E(G)\text{ and }a'b\notin E(G).$$
So
$$|V(D^*)|=nd=(q^4+q^3+q^2+q+1)(q^3+q^2+q+1)>q^7.$$
Fix a total order $\prec$ on $\PG(4,q)$. Extend $\prec$ to a total order on $V(D^*)$ such that $(a,b)$ precedes $(a',c)$ if $b\prec c$. Break ties arbitrarily when $b=c$.

\begin{definition}\label{ordered}
The graph $\Gamma_q=\Gamma_q^{\prec}$ has vertex set $V(D^*)$. If $(a,b)\prec(a',b')$ then $(a,b)$ and $(a',b')$ are adjacent in $\Gamma_q$ when $ab'\in E(G)$ and $a'b\notin E(G)$.
\end{definition}

\begin{lemma}[Brada\v{c} {\cite[Lemma 2.11]{Bradac26}}]\label{tournament}
The digraph $D^*$ contains no transitive tournament on $5$ vertices.
\end{lemma}

\begin{corollary}\label{clique}
The graph $\Gamma_q$ is $K_5$ free.
\end{corollary}

\begin{proof}
Suppose $f_1\prec f_2\prec\cdots\prec f_5$ are the vertices of a $K_5$ in $\Gamma_q$. Definition \ref{ordered} gives $(f_i,f_j)\in A(D^*)$ for every $i<j$. So these $5$ vertices contain a transitive tournament in $D^*$, contradicting Lemma \ref{tournament}.
\end{proof}

\subsection{Independence number}

Let $\cS$ be an independent set in $\Gamma_q$. For each $b\in\PG(4,q)$ let $A_b=\{a\in H_b:(a,b)\in\cS\}$. Then
$$|\cS|=\sum_{b\in\PG(4,q)}|A_b|.$$

\begin{lemma}\label{incidence}
If $b\prec c$ then
$$A_b\cap H_c\neq\emptyset\implies A_c\subseteq H_b.$$
\end{lemma}

\begin{proof}
Choose $a\in A_b\cap H_c$. If $a'\in A_c\setminus H_b$ then $(a,b)\prec(a',c)$, $ac\in E(G)$ and $a'b\notin E(G)$. The two vertices are adjacent in $\Gamma_q$, contradiction.
\end{proof}

\begin{lemma}\label{small}
We have
$$\sum_{\substack{c\in\PG(4,q)\\ A_b\cap H_c=\emptyset\text{ for every }b\prec c}}|A_c|\le n$$
and
$$\sum_{\substack{c\in\PG(4,q)\\ |A_c|\le1\\ A_b\cap H_c\neq\emptyset\text{ for some }b\prec c}}|A_c|\le n.$$
\end{lemma}

\begin{proof}
If $c\prec e$ both occur in the first sum then $A_c\cap H_e=\emptyset$. Since $A_e\subseteq H_e$ we get $A_c\cap A_e=\emptyset$. So these sets are pairwise disjoint and their total size is at most $n$. The second sum has at most $n$ terms, each at most $1$.
\end{proof}

Let
$$\cL=\{c\in\PG(4,q):|A_c|\ge2\text{ and }A_b\cap H_c\neq\emptyset\text{ for some }b\prec c\}.$$
Fix $c\in\cL$ and choose $b\prec c$ with $A_b\cap H_c\neq\emptyset$. Lemma \ref{incidence} gives $A_c\subseteq H_b\cap H_c$. Now $b\neq c\implies H_b\neq H_c$, so $H_b\cap H_c$ is a projective plane. Write $\langle A_c\rangle$ for the projective subspace spanned by $A_c$ and let
$$P_c=\{z\in\PG(4,q):A_c\subseteq H_z\},$$
the polar of $\langle A_c\rangle$. Since $|A_c|\ge2$ and $\langle A_c\rangle\subseteq H_b\cap H_c$, the subspace $\langle A_c\rangle$ is a line or a plane, so $P_c$ is a plane or a line, respectively. A line has only $q+1$ points, so $\langle A_c\rangle$ is a plane whenever $|A_c|\ge q+2$. Hence
$$|A_c|\le q^2+q+1,\,|P_c|\le q^2+q+1,\,|P_c|=q+1\text{ if }|A_c|\ge q+2.$$
Also every $b\prec c$ with $A_b\cap H_c\neq\emptyset$ lies in $P_c$ by Lemma \ref{incidence}.

\begin{lemma}\label{multiplicity}
Let $a\in\PG(4,q)$ and let $c_1\prec c_2\prec\cdots\prec c_r$ be elements of $\cL$ with $a\in A_{c_i}$ for every $i$. Then $r\le|P_{c_r}|$.
\end{lemma}

\begin{proof}
For $i<r$ we have $a\in A_{c_i}\cap H_{c_r}$ since $A_{c_r}\subseteq H_{c_r}$. Lemma \ref{incidence} gives $A_{c_r}\subseteq H_{c_i}$ so $c_i\in P_{c_r}$. Also $c_r\in P_{c_r}$. Hence $c_1,\ldots,c_r$ are $r$ distinct points of $P_{c_r}$.
\end{proof}

\begin{lemma}\label{mixing}
Let $B$ be the adjacency matrix of $G(4,q)$. For all $\bx,\by\in\R^{\PG(4,q)}$,
$$\left|\bx^{\T}B\by-\frac dn\left(\sum_ax_a\right)\left(\sum_by_b\right)\right|\le q^{3/2}\norm{\bx}_2\norm{\by}_2.$$
\end{lemma}

\begin{proof}
In the proof of Lemma 2.1 of \cite{Bradac26}, the nontrivial eigenvalues of $G(t,q)$ have absolute value $\sqrt{d-\lambda}$, where
$$d=\frac{q^t-1}{q-1},\,\lambda=\frac{q^{t-1}-1}{q-1}.$$
Taking $t=4$ gives $d=q^3+q^2+q+1$ and $\lambda=q^2+q+1$, so every nontrivial eigenvalue of $B$ has absolute value $q^{3/2}$. Write
$$\bx=\frac{\sum_ax_a}{n}\one+\bx_0,\,\by=\frac{\sum_by_b}{n}\one+\by_0$$
with $\bx_0,\by_0\perp\one$. Since $B\one=d\one$ we have
$$\bx^{\T}B\by=\frac dn\left(\sum_ax_a\right)\left(\sum_by_b\right)+\bx_0^{\T}B\by_0.$$
Since every eigenvalue of $B$ on $\one^{\perp}$ has absolute value $q^{3/2}$,
$$|\bx_0^{\T}B\by_0|\le q^{3/2}\norm{\bx_0}_2\norm{\by_0}_2\le q^{3/2}\norm{\bx}_2\norm{\by}_2.$$
\end{proof}

\begin{proposition}\label{large}
We have
$$\sum_{c\in\cL}|A_c|\le96(q+1)^2(q^2+q+1)+24(2+\sqrt2)(q^2+q)^{3/2}(q^2+q+1)^{1/2}.$$
\end{proposition}

\begin{proof}
For each power $k$ of $2$ with $2\le k\le q^2+q+1$, let $\cC_k=\{c\in\cL:k\le|A_c|<2k\}$, say $r=|\cC_k|$, and put
$$p_k=\begin{cases}q+1&\text{if }k\ge q+2,\\ q^2+q+1&\text{if }k\le q+1.\end{cases}$$
These classes cover $\cL$ and the bounds above give $|P_c|\le p_k$ for every $c\in\cC_k$. We first show that
$$r<12(q+1)p_k+6(q+1)q^{3/2}\sqrt{p_k/k}.$$
This is immediate when $r\le1$, so assume $r\ge2$. In the order induced by $\prec$, let $X$ be the first $\lfloor r/2\rfloor$ elements of $\cC_k$ and let $Y$ be the remaining $\lceil r/2\rceil$ elements. Define $x_a=|\{b\in X:a\in A_b\}|$. Then
$$\frac{kr}{3}\le k\lfloor r/2\rfloor\le\sum_ax_a=\sum_{b\in X}|A_b|\le2k\lfloor r/2\rfloor\le kr.$$
Applying Lemma \ref{multiplicity} to the $b\in X$ with $a\in A_b$ gives $x_a\le p_k$. Hence
$$\sum_ax_a^2\le p_k\sum_ax_a\le p_kkr.$$
Count incidences with multiplicity between the points in the sets indexed by $X$ and the hyperplanes indexed by $Y$:
$$I=\sum_ax_a|\{c\in Y:a\in H_c\}|=\sum_{b\in X}\sum_{c\in Y}|A_b\cap H_c|.$$
Apply Lemma \ref{mixing} with $\bx=(x_a)$ and $\by=\one_Y$, where $\one_Y$ is the indicator vector of $Y$. Since $d/n\ge1/(q+1)$ and $r/2\le\lceil r/2\rceil\le r$,
$$I\ge\frac1{q+1}\left(\sum_ax_a\right)\lceil r/2\rceil-q^{3/2}\sqrt{\left(\sum_ax_a^2\right)\lceil r/2\rceil}\ge\frac{kr^2}{6(q+1)}-q^{3/2}r\sqrt{p_kk}.$$
For $b\in X$ and $c\in Y$ we have $b\prec c$, so $A_b\cap H_c\neq\emptyset$ implies $b\in P_c$. Thus for each $c\in Y$ at most $p_k$ choices of $b$ contribute to $I$, and every contribution is less than $2k$. Therefore
$$I<2kp_k\lceil r/2\rceil\le2kp_kr.$$
Combining the two bounds for $I$ and dividing by $kr$ gives the bound on $r$. It follows that
$$\sum_{c\in\cC_k}|A_c|\le2kr<24(q+1)p_kk+12(q+1)q^{3/2}\sqrt{p_kk}.$$
For powers of $2$ with $k\le K$ we have $\sum_kk<2K$ and $\sum_k\sqrt k<(2+\sqrt2)\sqrt K$. Taking $K=q+1$ for $k\le q+1$ and $K=q^2+q+1$ for $k\ge q+2$, the sum over each of these two ranges satisfies
$$\sum_kp_kk<2(q+1)(q^2+q+1),\,\sum_k\sqrt{p_kk}<(2+\sqrt2)\sqrt{(q+1)(q^2+q+1)}.$$
Adding the two ranges proves the result.
\end{proof}

\begin{theorem}\label{independence}
For every prime power $q\ge2$,
$$\alpha(\Gamma_q)\le2n+96(q+1)^2(q^2+q+1)+24(2+\sqrt2)(q^2+q)^{3/2}(q^2+q+1)^{1/2}<600q^4.$$
\end{theorem}

\begin{proof}
Every $c\in\PG(4,q)$ is counted in one of the two sums in Lemma \ref{small}, or belongs to $\cL$. So Lemma \ref{small} and Proposition \ref{large} give the first inequality. For $q\ge2$,
$$n\le\frac{31}{16}q^4,\,q+1\le\frac32q,\,q^2+q+1\le\frac74q^2.$$
So
$$|\cS|\le\left(\frac{31}{8}+378+18\sqrt{21}\,(1+\sqrt2)\right)q^4<582q^4<600q^4.$$
\end{proof}

\subsection{The Ramsey bound}

\begin{proof}[Proof of Theorem \ref{explicit}]
By the definition of $V(D^*)$,
$$|V(\Gamma_q)|=nd=(q^4+q^3+q^2+q+1)(q^3+q^2+q+1).$$
The clique bound is Corollary \ref{clique} and the independence bound is Theorem \ref{independence}.\vspace{0.5em} \\
Let $t$ be sufficiently large and put $X=(t/9600)^{1/4}$. Choose the least power $q$ of $2$ with $q\ge X$. Then $q<2X$, so
$$\alpha(\Gamma_q)<600q^4<600(2X)^4=t.$$
Also
$$|V(\Gamma_q)|>q^7\ge X^7=\left(\frac{t}{9600}\right)^{7/4}.$$
Since $\Gamma_q$ is $K_5$ free,
$$r(5,t)>\left(\frac{t}{9600}\right)^{7/4}$$
for all sufficiently large $t$. Hence $r(5,t)=\Omega(t^{7/4})$.
\end{proof}

\section{The algorithmic bound}\label{deterministic}

\subsection{The construction}

Following Kostochka, Pudl\'ak and R\"odl \cite{KPR}, let $V(H(G))=E(G)$. Join disjoint edges $e,f$ in $H(G)$ when an edge of $G$ joins their endpoint sets. We get $G$ by pruning the graph below.\vspace{0.5em} \\
Let $q=3^h$ with $3\nmid h$ and $f(X)=X^{14}$. Coulter and Matthews \cite[Theorem 4.1]{CM} proved that $X\mapsto f(X+c)-f(X)$ permutes $\F_q$ for every $c\ne0$. Let $V(B)=\F_q^2$ and
$$(x,y)(x',y')\in E(B)\iff x\ne x'\text{ and }y+y'=f(x+x').$$

\begin{lemma}\label{ambient}
We have $C_4\nsubseteq B$, $|V(B)|=q^2$, $\Delta(B)=q-1$ and $\norm{B-J/q}\le\sqrt q+1$.
\end{lemma}

\begin{proof}
Let $A$ be the adjacency matrix with $x=x'$ and loops allowed. A common neighbor $(a,b)$ of distinct vertices $(x,y),(x',y')$ satisfies
$$f(a+x)-f(a+x')=y-y',\,b=f(a+x)-y.$$
There is $1$ solution for $x\ne x'$ and $0$ for $x=x'$. So $C_4\nsubseteq B$. Also $A\one=q\one$ and for $z\perp\one$,
$$\norm{Az}_2^2=q\norm z_2^2-\sum_x\left(\sum_yz_{x,y}\right)^2\le q\norm z_2^2.$$
Hence $\norm{A-J/q}\le\sqrt q$. The involution $(x,y)\mapsto(x,f(2x)-y)$ has matrix $A-B$ so
$$B\one=(q-1)\one,\,\norm{B-J/q}\le\norm{A-J/q}+\norm{A-B}\le\sqrt q+1.$$
\end{proof}

\subsection{Petersen counts}

Let $P$ be the Petersen graph with edges $a_ia_{i+1}$, $a_ib_i$ and $b_ib_{i+2}$ with indices modulo $5$. A Petersen embedding maps its $10$ labeled vertices injectively into the ambient graph and preserves its $15$ edges. It doesn't need to be induced.

\begin{proposition}\label{pruning}
For a constant $C$ we construct in polynomial time $B_0\subseteq B$ with
$$n=|V(B_0)|=q^2-O(q),\,\Delta(B_0)\le q,\,\norm{B_0-J/q}\le\sqrt q+1+C.$$
Every edge lies in at most $Cq^2$ embeddings. Mark each $P_3$ with more than $Cq$ embeddings. Each edge belongs to at most $C$ marked $P_3$.
\end{proposition}

\begin{proof}
An oriented edge has endpoints $(x_0,y_0)$ and $(x_1,f(x_0+x_1)-y_0)$. A $P_3$ centered at $(x_0,y_0)$ adds the other coordinate $x_2$. Lemma \ref{rooted} gives nonzero polynomials $R_e(x_0,y_0,x_1)$ and $R_p(x_0,y_0,x_1,x_2)$ with extension bounds $O(q^2)$ at $R_e\ne0$ and $O(q)$ at $R_p\ne0$.\vspace{0.5em} \\
Take a nonzero coefficient $c$ of $R_p$ in $x_2$ and let
$$R(x_0,y_0,x_1)=(cR_e)(x_0,y_0,x_1)(cR_e)(x_1,f(x_0+x_1)-y_0,x_0).$$
Edge reversal is invertible so $R\ne0$. At $R\ne0$ the edge has $O(q^2)$ embeddings and at most $2\deg_{x_2}R_p$ partners with $R_p=0$. Choose $C$ larger than these constants and $\deg_{x_1}R$.\vspace{0.5em} \\
In $B$ mark each $P_3$ with more than $Cq$ embeddings. Mark an edge with more than $Cq^2$ embeddings or more than $C$ marked $P_3$ through it. Every marked edge satisfies $R=0$. If $R(x_0,y_0,X)\not\equiv0$ then at most $\deg_{x_1}R\le C$ marked edges meet $(x_0,y_0)$. The vertices with $R(x_0,y_0,X)\equiv0$ lie on the zeros of a nonzero coefficient of $R$, so there are $O(q)$ of them.\vspace{0.5em} \\
Delete vertices meeting more than $C$ marked edges, then delete the remaining marked edges. The first step removes $O(q)$ vertices. The second removes a graph of maximum degree at most $C$. Its adjacency matrix has norm at most $C$ so $\norm{B_0-J/q}\le\sqrt q+1+C$. Counts only decrease. We test $O(q^{20})$ vertex assignments.
\end{proof}

Let $V(\Gamma)=E(B_0)$. Join edges forming a marked $P_3$ or lying in a common $C_3$. Since an edge in a $C_4$ free graph lies in at most $1$ triangle, $\Delta(\Gamma)\le C+2$. Let $\mathcal P$ be the Petersen embeddings whose edges are independent in $\Gamma$.

\begin{lemma}\label{extensions}
For $F\subseteq E(B_0)$ with $|F|=j$,
$$|\{P\in\mathcal P:F\subseteq E(P)\}|=\begin{cases} O(q^2)&j=1,\\ O(q)&2\le j\le5,\\ O(1)&6\le j\le15. \end{cases}$$
A fixed vertex lies in $O(q^3)$ embeddings in $\mathcal P$.
\end{lemma}

\begin{proof}
Two nonadjacent vertices of $P$ have a unique common neighbor. Since $C_4\nsubseteq B_0$, their images determine at most $1$ image for that neighbor. Repeat this rule until no further vertex is determined.\vspace{0.5em} \\
Let $S\subsetneq V(P)$ be unchanged by the rule with $|S|\ge3$. Each vertex outside $S$ has at most $1$ neighbor in $S$. Since $P$ has degree $3$ and girth $5$,
$$\delta(P-S)\ge2\implies|S|\le5,\,\operatorname{diam}P[S]\le2\implies P[S]\in\{P_3,K_{1,3},C_5\}.$$
For a fixed assignment of $F$ to Petersen edges let $S$ be all vertices determined by the rule. Then
$$F\cong2K_2\implies P[S]\in\{C_5,P\}.$$
$$3\le|F|\le5\implies P[S]\in\{K_{1,3},C_5,P\}.$$
$$|F|\ge6\implies S=V(P).$$
For a fixed $K_{1,3}$ or $C_5$ choose an adjacent outside vertex in at most $\Delta(B_0)\le q$ ways. Since $K_{1,3}\nsubseteq C_5$ and $|V(C_5)|+1=6>5$, the rule now determines all $10$ vertices. A marked $P_3$ has $0$ extensions in $\mathcal P$ and an unmarked $P_3$ has at most $Cq$. There are at most $30^{|F|}$ edge label assignments. Proposition \ref{pruning} covers $|F|=1$, and
$$|\{P\in\mathcal P:v\in V(P)\}|\le\sum_{e\ni v}|\{P\in\mathcal P:e\in E(P)\}|\le\Delta(B_0)Cq^2\le Cq^3.$$
\end{proof}

\subsection{Selecting edges}

Let $\xi_e$ be independent $0,1$ variables with $\Prob(\xi_e=1)=p$. Retain $e$ when $\xi_e=1$ and every neighbor in $\Gamma$ has value $0$. Its indicator is $T_e=\xi_e\prod_{f\in N_\Gamma(e)}(1-\xi_f)$. The graph $T$ contains no $C_3$, $C_4$ or marked $P_3$, and
$$\E T_e\ge p\bigl(1-(C+2)p\bigr),\,\Prob(F\subseteq E(T))\le p^{|F|}.$$
For $P\in\mathcal P$ let $I_P=\prod_{e\in E(P)}T_e$. Let $M_{uv}=\sum_{P\in\mathcal P:\,uv\in E(P)}I_P$. Put $X=T-\E T$ and $Y=M-\E M$.

\begin{lemma}\label{moments}
For fixed $k\ge1$ with $p\le q^{-1/7}$, $pq\ge1$ and $q^3p^{15}\ge1$,
$$\E\tr X^{2k}=O\bigl(n(pq)^k\bigr),\,\E\tr Y^{2k}=O\bigl(n(q^3p^{15})^k\bigr).$$
\end{lemma}

\begin{proof}
If $T_e,T_f$ use a common $\xi$, then $e,f$ have distance at most $2$ in $\Gamma$. For fixed $e$ there are at most $1+(C+2)^2$ choices for $f$.\vspace{0.5em} \\
Expand $\tr Y^{2k}$ over walks $v_0v_1\cdots v_{2k}$ with $v_{2k}=v_0$ and embeddings $P_i$ containing $v_{i-1}v_i$. The centered expansion gives
$$\left|\E\prod_{i=1}^{2k}(I_{P_i}-\E I_{P_i})\right|\le2^{2k}p^{|\bigcup_iE(P_i)|}.$$
For each $S\subseteq\{1,\ldots,2k\}$ its term is bounded by $p^{|\bigcup_{i\in S}E(P_i)|+15(2k-|S|)}\le p^{|\bigcup_iE(P_i)|}$.
Join $i,j$ when $I_{P_i},I_{P_j}$ use a common variable. A singleton component contributes $\E(I_{P_i}-\E I_{P_i})=0$. So the number $c$ of components in a nonzero term satisfies $2c\le2k$.\vspace{0.5em} \\
Fix the partition of $\{1,\ldots,2k\}$ into these components and start with the component containing $P_1$. After completing a component, if any remain, choose $i$ with $P_i$ completed and $P_{i+1}$ not yet chosen, with indices modulo $2k$. Such an $i$ exists along the cyclic walk. Start the next component at $P_{i+1}$. Its edge $v_iv_{i+1}$ has $v_i$ already fixed. Within each component choose each later occurrence to share a variable with an earlier one. This is possible because the component is connected. Only $v_0$ has $n$ choices.\vspace{0.5em} \\
Each embedding has $30$ oriented walk edges. Choosing one fixes its two walk vertices. Different components share no edge. For a later occurrence let $j$ count its edges among those already chosen. Lemma \ref{extensions} gives
$$\begin{array}{c|c|c} \text{occurrence}&\text{choices}&\text{new edges}\\ \hline \text{first in a component}&O(q^3)&15\\ j=0&O(q^2)&15\\ j=1&O(q^2)&14\\ 2\le j\le5&O(q)&15-j\\ 6\le j\le15&O(1)&15-j \end{array}$$
For $j=0$ the shared variable fixes an edge within distance $2$ of an earlier edge. For $j\ge1$ choose the shared edges from at most $30k$ exposed edges. Since $q^2p^{14}\le1$ and $qp^{10}\le1$, every later occurrence in a component contributes $O(1)$ after multiplying by $p$ for each new edge. For fixed $k$ there are $O(1)$ partitions, orders and orientations. So
$$\E\tr Y^{2k}\le O(n)\sum_{c=1}^k(q^3p^{15})^c=O\bigl(n(q^3p^{15})^k\bigr).$$
For $X$ use single edges. A component starts with at most $q$ choices and a factor $p$. Every later edge has $O(1)$ choices. So $\E\tr X^{2k}\le O(n)\sum_{c=1}^k(pq)^c=O(n(pq)^k)$.
\end{proof}

\subsection{The algorithm}\label{algorithm}

Choose a fixed integer $K$ with $C3^{-14K}\le1/4$ and let $p=3^{-K-\lceil h/7\rceil}=\Theta(q^{-1/7})$. Start at $h_0$ large enough that $(C+2)p\le1/4$, $n\ge q^2/2$, $pq\ge1$ and $q^3p^{15}\ge1$.\vspace{0.5em} \\
Let $G$ be $T$ with every edge in a Petersen copy deleted. Then $C_3,C_4,P\nsubseteq G$. Entrywise $G\ge T-M$ and
$$\E M\le Cq^2p^{15}B_0,\,\E T-\E M\ge\bigl[p(1-(C+2)p)-Cq^2p^{15}\bigr]B_0\ge\frac p2B_0.$$
So for $z\ge0$
$$z^{\T}Gz\ge\frac p2z^{\T}B_0z-(\norm X+\norm Y)\norm z_2^2.$$
Fix $k\ge3$. Let $R_q$ be the least nonnegative integer with $R_q^{2k}\ge\E\tr X^{2k}+\E\tr Y^{2k}$. Fix each $\xi_e$ to minimize the conditional expectation of $\tr X^{2k}+\tr Y^{2k}$. The means $\E T,\E M$ stay fixed. The final sum doesn't exceed its initial expectation so
$$\norm X+\norm Y\le2\bigl(\tr X^{2k}+\tr Y^{2k}\bigr)^{1/(2k)}\le2R_q=O(q^{3/7+1/k}).$$
For an independent set $S\ne\varnothing$, take $z=\one_S$. Then
$$0\ge\frac p2\left(\frac{|S|^2}{q}-(\sqrt q+1+C)|S|\right)-2R_q|S|,$$
so $\alpha(G)\le q(\sqrt q+1+C+4R_q/p)$. Let
$$t_q=\left\lceil q\left(\lceil\sqrt q\rceil+1+C+\frac{4R_q}{p}\right)\right\rceil.$$
Given $t\ge2$ test $q=3^h$ in increasing order with $h\ge h_0$ and $3\nmid h$. Stop at the first $t_q\ge t$ and use the preceding field. If none exists return one vertex. Otherwise fix the variables as above and return $H(G)$.

\begin{lemma}\label{runtime}
For fixed $k$ the algorithm is deterministic and polynomial in $t$.
\end{lemma}

\begin{proof}
For disjoint sets $S_1,S_0$ of unfixed variables,
$$\E\left[\prod_{e\in S_1}\xi_e\prod_{e\in S_0}(1-\xi_e)\right]=p^{|S_1|}(1-p)^{|S_0|}.$$
If a variable must be both $0$ and $1$ the term is $0$. Each trace term uses at most $30k(C+3)$ variables. Expanding over walks, embeddings and centered factors gives at most
$$n^{2k}(n^{10})^{2k}2^{2k}=q^{O(k)}$$
terms. Their common denominator divides $3^{O(kh)}$ and their numerators have $O(kh)$ binary digits. So all conditional expectations are evaluated in $q^{O(k)}$ time. We find $R_q$ by binary search using its inequality.\vspace{0.5em} \\
Since $t_q\ge q^{3/2}$ and consecutive allowed field orders differ by at most $9$, the search uses $O(\log t)$ fields and stops at $q=O(t^{2/3})$. To construct $\F_{3^h}$ we enumerate monic polynomials and test irreducibility in polynomial time in $q$. We compute the constants in Subsection \ref{polynomials} once before starting the search. All graph operations are polynomial in $q$.
\end{proof}

\subsection{The Ramsey bound}

\begin{lemma}\label{superline}
If $C_3,C_4,P\nsubseteq G$ then $\omega(H(G))\le4$ and $\alpha(H(G))\le\alpha(G)$.
\end{lemma}

\begin{proof}
A $K_5\subseteq H(G)$ gives $5$ disjoint edges of $G$ with an edge between each pair. An endpoint can't meet both ends of another matching edge. Suppose $v$ meets $u_1,u_2,u_3$ in distinct matching edges $u_iu_i'$. For $i<j$ the absence of $C_3,C_4$ gives
$$u_iu_j,u_iu_j',u_i'u_j\notin E(G)\implies u_i'u_j'\in E(G).$$
So $u_1',u_2',u_3'$ form a $C_3$, contradiction. Each endpoint therefore has at most $2$ neighbors in the other matching edges. The $\binom52=10$ pairs each need an edge so equality holds at every endpoint. Together with the matching these edges form a cubic graph on $10$ vertices of girth at least $5$.\vspace{0.5em} \\
The distance layers from any vertex have sizes $1,3,6$. Every vertex in the last layer has degree $2$ within that layer so the $6$ vertices induce a $C_6$. Vertices with the same parent must be opposite on this $C_6$ to avoid $C_3,C_4$. These adjacencies give $P\subseteq G$, contradiction.\vspace{0.5em} \\
An independent $F\subseteq V(H(G))$ selects edges of $G$ containing no $P_4$ or $C_3$. Their components are stars with no edge of $G$ between distinct components. Take all leaves of each star with at least $2$ edges and one endpoint of each single edge. These vertices are independent and number $|F|$. So $\alpha(G)\ge|F|$.
\end{proof}

\begin{proof}[Proof of Theorem \ref{algorithmic}]
Take $z=\one$ in Subsection \ref{algorithm}. Then
$$2|E(G)|\ge\frac p2\left(\frac{n^2}{q}-(\sqrt q+1+C)n\right)-2R_qn.$$
Since $p=\Theta(q^{-1/7})$, $n=q^2-O(q)$ and $k\ge3$,
$$|V(H(G))|=|E(G)|\ge\frac{pq^3}{16}-O(q^{33/14}+q^{17/7+1/k})=\Omega(q^{20/7}).$$
$$\alpha(H(G))\le t_q=O(q^{11/7+1/k}).$$
Let $Q$ be the first tested field order with $t_Q\ge t$. For large $t$ its predecessor $q$ exists and $t_q<t$. So
$$t\le t_Q=O(Q^{11/7+1/k}),\,q\ge Q/9=\Omega\bigl(t^{7k/(11k+7)}\bigr).$$
Lemmas \ref{runtime} and \ref{superline} give polynomial running time, $\omega(H(G))\le4$ and $\alpha(H(G))<t$. Also
$$|V(H(G))|=\Omega\bigl(t^{20k/(11k+7)}\bigr),\,\frac{20}{11}-\frac{20k}{11k+7}=\frac{140}{11(11k+7)}\longrightarrow0.$$
Choose $k$ sufficiently large for the given $\varepsilon$.
\end{proof}

\subsection{Polynomial estimates}\label{polynomials}

Coordinates are algebraically independent over $\F_3$ if no nonzero polynomial over $\F_3$ vanishes at them.

\begin{lemma}\label{elimination}
Say a fixed polynomial system over $\F_3$ has at most $d$ algebraically independent coordinates in every solution over any extension field. Fixing $r\le d$ polynomial values leaves $O(q^{d-r})$ solutions over $\F_q$ outside the zeros of a nonzero polynomial in those values.
\end{lemma}

\begin{proof}
For a condition $f\ne0$ add $zf=1$, where $z=1/f$ is uniquely determined. Let $m$ be the number of coordinates, including any added this way. Write all equations as $f_i=0$, including those fixing the polynomial values. Work over $\F_3(t_1,\ldots,t_r)$ with independent variables $t_1,\ldots,t_r$. If $1=\sum_i h_if_i$, multiply by a common denominator. The resulting nonzero polynomial in $t_1,\ldots,t_r$ must vanish at every solution.\vspace{0.5em} \\
Otherwise for each set $S$ of $d-r+1$ coordinates we have
$$0\ne\sum_i h_if_i\in\F_3(t_1,\ldots,t_r)[X_j:j\in S].$$
If not, the equations still can't express $1$ over $\F_3(t_1,\ldots,t_r,X_j:j\in S)$. They then have a solution in an extension of this field. Since the $t_i$ are polynomial functions of the solution coordinates their field contains $r+|S|=d+1$ algebraically independent elements, contradiction. We use the polynomial test and algorithm in \cite[Chapters 1 and 2]{CLO}.\vspace{0.5em} \\
Order monomials first by total degree and then lexicographically. Clear all denominators above and let $R$ be their product times the resulting leading coefficients. If $R\ne0$ after fixing the parameter values then no leading coefficient vanishes, so the leading monomials stay the same. Choose $M$ larger than every exponent in them. \vspace{0.5em} \\
Add $X_i^q-X_i=0$ for $1\le i\le m$. After division, every remaining monomial $X_1^{a_1}\cdots X_m^{a_m}$ satisfies
$$0\le a_i<q,\,|\{i:a_i\ge M\}|\le d-r.$$
Indeed if $a_i\ge M$ for $d-r+1$ coordinates, the leading monomial of the relation on those coordinates divides $X_1^{a_1}\cdots X_m^{a_m}$, so that monomial reduces more. \vspace{0.5em} \\
For each solution $a\in\F_q^m$ the polynomial $\prod_{i=1}^m\bigl(1-(X_i-a_i)^{q-1}\bigr)$ is $1$ at $a$ and $0$ at every other solution. Its remainder has the same values. These remainders are linearly independent, so
$$\bigl|\{a\in\F_q^m:f_i(a)=0\text{ for all }i\}\bigr|\le\sum_{j=0}^{\min(d-r,m)}\binom mj M^{m-j}q^j=O(q^{d-r}).$$
\end{proof}

Write $a_i=(x_i,y_i)$ and $b_i=(z_i,f(x_i+z_i)-y_i)$. The outer $C_5$ determines $y_1,\ldots,y_5$ by
$$2y_1=f(x_1+x_2)-f(x_2+x_3)+f(x_3+x_4)-f(x_4+x_5)+f(x_5+x_1).$$
$$y_{i+1}=f(x_i+x_{i+1})-y_i\text{ for }1\le i\le4.$$
Let $F_{ij}=f(z_i+z_j)-f(x_i+z_i)-f(x_j+z_j)+y_i+y_j$. For the inner edges we must have $F_{13}=F_{14}=F_{35}=F_{24}=F_{25}=0$.

\begin{lemma}\label{rooted}
There are nonzero polynomials $R_e(x_0,y_0,x_1)$ and $R_p(x_0,y_0,x_1,x_2)$ giving $O(q^2)$ Petersen embeddings through a fixed edge at $R_e\ne0$ and $O(q)$ through a fixed $P_3$ at $R_p\ne0$.
\end{lemma}

\begin{proof}
A common neighbor $(a,b)$ of $(x,y),(x,y')$ gives $y+b=f(x+a)=y'+b$, hence $y=y'$. Adjacent vertices of $P$ have different first coordinates in $B$ and every nonadjacent pair shares a neighbor. So all $10$ first coordinates of an embedding are distinct. In particular
$$(z_1-x_3)(z_1-x_4)(z_3-x_5)(z_4-x_2)\ne0.$$
In order $F_{13},F_{14},F_{35},F_{24}$ have degree $13$ in $z_3,z_4,z_5,z_2$ with nonzero leading coefficients $2(z_1-x_3)$, $2(z_1-x_4)$, $2(z_3-x_5)$ and $2(z_4-x_2)$. So those coordinates are algebraic over $x_1,\ldots,x_5,z_1$.\vspace{0.5em} \\
For $\deg g=m$ and $\deg h=n$, the resultant $\Res_X(g,h)$ is the determinant whose rows are the coefficients of $X^{n-1}g,\ldots,g,X^{m-1}h,\ldots,h$ in descending powers. A common root makes it $0$ \cite[Chapter 3]{CLO}. Let
$$R_5(z_5)=\Res_{z_3}(F_{13},F_{35}).$$
$$R_2(z_2)=\Res_{z_4}(F_{14},F_{24}).$$
$$R_{25}(z_2)=\Res_{z_5}(R_5,F_{25}).$$
Every solution satisfies $\Res_{z_2}(R_2,R_{25})=0$. Since $F_{13}$ doesn't involve $z_5$ and $F_{14},R_5$ don't involve $z_2$, we have $\deg R_5,\deg R_2\le13^2=169$ and $\deg R_{25}\le13^3=2197$.
Let $a\in\F_{81}$ satisfy the irreducible polynomial $a^4+a+2=0$. Code was used to show that at $(x_1,x_2,x_3,x_4,x_5,z_1)=(0,1,a,a^2,a^3,2)$
$$\deg R_5=\deg R_2=169,\,\deg R_{25}=2197.$$
$$\Res_{z_2}(R_2,R_{25})=1+2a+2a^2+2a^3\ne0.$$
The degrees hit their bounds. So substitution preserves the determinant sizes and the last resultant is a relation on $x_1,\ldots,x_5,z_1$. At most $5$ coordinates are algebraically independent.\vspace{0.5em} \\
An oriented edge fixes $3$ polynomial values and an ordered $P_3$ fixes $4$. Apply Lemma \ref{elimination} with $d=5$ to each of the $30$ oriented edges and $60$ ordered $P_3$ in $P$. Let $R_e,R_p$ be the products of the corresponding nonzero polynomials. At $R_e\ne0$ the total is $30O(q^{5-3})=O(q^2)$. At $R_p\ne0$ it is $60O(q^{5-4})=O(q)$.
\end{proof}
As a concluding remark, we believe planar functions may have some properties yet to fully exploit in Ramsey Theory.
\section*{Acknowledgments}
We are grateful to Dr. Jacques Verstraete for his valuable discussions, guidance, and reviews of some writeups.

\section{Note on AI Usage}

For the constructive result, the authors suggested $D^*(3,q)$ as a starting point for $r(4,t)$. ChatGPT proposed a construction matching the exponent $5/3$ in \cite{HIVM26} and checked small cases. The authors worked out all mathematical proofs for the constructive result. This was easily adapted to $r(5,t)$.\vspace{0.5em} \\
For the algorithmic result, ChatGPT proposed the algorithm and ran computational tests. Around half of the mathematical work for the algorithm was done by the authors, with AI assistance in the proofs in Subsection 3.3. AI assistance was also used in typsetting the final paper.

\end{document}